\documentclass[10pt,twoside]{article}
\usepackage{amssymb,amsmath,amsthm,latexsym}
\usepackage{amsfonts}
    \usepackage{amsfonts}
\usepackage{graphicx}

\newtheorem{theorem}{Theorem}[section]

\newtheorem{conjecture}[theorem]{Conjecture}
\newtheorem{corollary}[theorem] {Corollary}
\newtheorem{definition}[theorem]{Definition}

\newtheorem{lemma} [theorem]{Lemma}

\begin{document}

\title{On Distance Antimagic Graphs}
\author{
\normalsize Rinovia Simanjuntak and Kristiana Wijaya\\
\vspace{-1mm}
\normalsize Combinatorial Mathematics Research Group\\
\vspace{-1mm}
\normalsize Faculty of Mathematics and Natural Sciences\\
\vspace{-1mm}
\normalsize Institut Teknologi Bandung, Bandung 40132, Indonesia\\
\vspace{-2mm}
\scriptsize {\small {\bf e-mail}: {\tt rino@math.itb.ac.id, kristiana$\_$wijaya@yahoo.com}}\\\\
}
\date{}
\maketitle

\vspace*{-4mm}
\begin{abstract}
For an arbitrary set of distances $D\subseteq \{0,1, \ldots, diam(G)\}$, a $D$-weight of a vertex $x$ in a graph $G$ under a vertex labeling $f:V\rightarrow \{1,2, \ldots , v\}$ is defined as $w_D(x)=\sum_{y\in N_D(x)} f(y)$, where $N_D(x) = \{y \in V| d(x,y) \in D\}$. A graph $G$ is said to be $D$-distance magic if all vertices has the same $D$-vertex-weight, it is said to be $D$-distance antimagic if all vertices have distinct $D$-vertex-weights, and it is called $(a,d)-D$-distance antimagic if the $D$-vertex-weights constitute an arithmetic progression with difference $d$ and starting value $a$.

In this paper we study some necessary conditions for the existence of $D$-distance antimagic graphs. We conjecture that such conditions are also sufficient. Additionally, we study $\{1\}$-distance antimagic labelings for some cycle-related connected graphs: cycles, suns, prisms, complete graphs, wheels, fans, and friendship graphs.
\end{abstract}

\section{Introduction}

\noindent As standard notation, assume that $G$=$G(V,E)$ is a finite, simple, and undirected graph with $v$ vertices and $e$ edges. By a {\em labeling} we mean a one-to-one mapping that carries a set of graph elements onto a set of numbers, called {\em labels}.\\

\noindent The notion of distance magic labeling was introduced separately in the PhD thesis of Vilfred \cite{Vi94} in 1994 and an article by Miller {\it et. al} \cite{MRS03} in 2003. A {\em distance magic labeling} is a bijection $f:V\rightarrow \{1,2, \ldots , v\}$ with the property that there is a constant {\sf k} such that at any vertex $x$, the {\em vertex-weight} of $x$, $w(x)=\sum_{y\in N(x)} f(y) ={\sf k}$, where $N(x)$ is the set of vertices adjacent to $x$. This labeling was introduced due to two different motivations; as a tool in utilizing magic squares into graphs and as a natural extension of previously known graph labelings: magic labeling \cite{Sed64,KR70} and radio labeling (which is distance-based) \cite{GY92}.\\

\noindent In the last decade, many results on distance magic labeling have been published. Several families of graphs have been showed to admit the labeling \cite{Vi94,Ji99,MRS03,ARSP04,Ra08,Be09,SAS09} and constructions of distance magic graphs have also been studied \cite{Fr07,FKK06,SFMRW09,FKK11,KS12}. It has also been showed that there is no forbidden subgraph characterization for distance magic graph \cite{Vi94,ARSP04,RSP04}. Additionally, an application of the labeling in designing incomplete tournament is introduced in \cite{FKK06}. For more results in distance magic labeling, please refer to Gallian's dynamic survey on graph labelings \cite{Ga}.\\

\noindent O'Neal and Slater \cite{OS11b,OS13} generalized the notion of distance magic labeling to an arbitrary set of distances $D\subseteq \{0,1, \ldots, diam(G)\}$, where $diam(G)$ is the diameter of $G$. As in the previous distance magic labeling, the domain of this new labeling is the set of all vertices and the codomain is $\{1,2, \ldots , v\}$. We define the {\em $D$-vertex-weight} of each vertex $x$ in $G$, $w_D(x)=\sum_{y\in N_D(x)} f(y)$, where $N_D(x) = \{y \in V| d(x,y) \in D\}$. If all vertices in $G$ have the same weight, we call the labeling a {\em $D$-distance magic labeling}.\\

\noindent Recently, Arumugam and Kamatchi \cite{AK12} considered an antimagic version of distance labeling. They defined an {\em $(a,d)$-distance antimagic labeling} of a graph $G$ as a bijection $f:V\rightarrow \{1,2, \ldots ,v\}$ such that the set of all vertex-weights is $\{a, a+d, a+2d, \ldots , a+(v-1)d\}$, where $a$ and $d$ are fixed integers with $d \geq 0$. Any graph which admits such a labeling is called an $(a,d)$-distance antimagic graph. The characterization of $(a,d)$-distance antimagic cycles and $(a,d)$-distance antimagic labelings for paths and prisms were also studied in \cite{AK12}. Froncek proved that disjoint copies of the Cartesian
product of two complete graphs and its complement are $(a,2)$-distance antimagic and $(a,1)$-distance antimagic, respectively (see \cite{Fr13} and \cite{Fr}). He also proved that disjoint copies of the hypercube $Q_3$ is $(a,1)$-distance antimagic. \\

\noindent In addition to the $(a,d)$-distance antimagic labeling, we also consider the following three other labelings.

\begin{definition}
Let $G$ be a graph, $x$ a vertex in $G$, $f$ a bijection from $V$ onto $\{1,2,\ldots,v\}$, and $D\subseteq \{0,1, \ldots, diam(G)\}$.

\noindent The bijection $f$ is called {\bf distance antimagic labeling} if all vertices have distinct vertex-weights. A graph is called {\bf distance antimagic} if it admits a distance antimagic labeling.

\noindent The bijection $f$ is called a {\bf $D$-distance antimagic labeling} if the $D$-vertex-weights are all different. The bijection $f$ is called an {\bf $(a,d)$-$D$-distance antimagic labeling} if all $D$-vertex-weights constitute an arithmetic progression with difference $d$ and starting value $a$, for $a$ and $d$ fixed integers with $d \geq 0$. A graph $G$ is {\bf $D$-distance antimagic} or {\bf $(a,d)$-$D$-distance antimagic} if it admits a $D$-distance antimagic labeling or an $(a,d)$-$D$-distance antimagic labeling, respectively.
\end{definition}

\noindent Note that if $D=\{1\}$, a $D$-distance antimagic labeling is a distance antimagic labeling and similarly an $(a,d)$-$D$-distance antimagic labeling is an $(a,d)$-distance antimagic labeling. If $d=0$, an $(a,0)$-$D$-distance antimagic labeling is a $D$-distance magic labeling. It is clear that if a graph is $(a,d)$-$D$-distance antimagic for $d>0$ then it is also $D$-distance antimagic, but not necessarily the other way around.\\

\noindent In this paper we study some necessary conditions for the existence of $D$-distance antimagic graphs. Additionally, we study distance antimagic labelings for some connected graphs containing one or more cycles: cycles, suns, prisms, complete graphs, wheels, fans, and friendship graphs. Finally, we conjecture that the necessary conditions for the existence of $D$-distance antimagic graphs are also sufficient.

\section{Main Result}

\noindent We start with a couple of obvious observations.
\begin{lemma}
If a graph contains two vertices with the same neighborhood then it is not distance antimagic.
\label{not-da}
\end{lemma}
\begin{proof}
If $G$ has two vertices with the same neighborhood, say $u$ and $v$, then $w(u)=w(v)$, a contradiction.
\end{proof}

\noindent Let us define a {\em $D$-neighborhood} of a vertex $x$ as the set of all vertices at distance $k$ to $x$, where $k \in D$. Then Lemma \ref{not-da} can be generalized in the following lemma.
\begin{lemma}
If a graphs contains two vertices with the same $D$-neighborhood then it is not $D$-distance antimagic.
\label{not-Dda}
\end{lemma}

\noindent As a consequence of Lemma \ref{not-da}, we have
\begin{corollary}
All complete multipartite graphs are not distance antimagic.
\label{Kmn}
\end{corollary}

\noindent The following lemma gives us an upper bound for $d$ of an $(a,d)$-distance antimagic labeling of a regular graph.

\begin{lemma}
Let $G$ be an $r$-regular graph.
If $G$ is $(a,d)$-distance antimagic then $d\leq r \frac{v-r}{v-1}$ and $a=\frac{r(v+1)-d(v-1)}{2}$.
\label{ubd}
\end{lemma}
\begin{proof}
\noindent If we consider a particular vertex $x$, it contributes exactly $d(x)$ times to the sum of all vertex-weights, where $d(x)$ is the degree of $x$. Thus,
\[a + (a+d) + \ldots + (v-1)d = \sum_{x\in V(G)} d(x) f(x),\]
which leads to
\[va + d \frac{v(n-1)}{2} = \sum_{x\in V(G)} d(x) f(x).\]
Since $G$ is an $r$-regular graph, then
\[va + d \frac{v(n-1)}{2} = r \sum_{x\in V(G)} f(x) = r \frac{v(v+1)}{2}.\]
Therefore, $d = \frac{r(v+1)-2a}{v-1}$ which gives us the second result.\\

\noindent Now, consider the least possible value of a vertex-weight. Obviously, it has to be equal to $1+2+\ldots+d$, and so $a\geq \frac{r(r+1)}{2}$. This gives the desired upper bound for $d$.
\[d \leq \frac{r(v+1)-2\frac{r(r+1)}{2}}{v-1} =  r \frac{v-r}{v-1}.\]
\end{proof}

\noindent Next we study distance antimagic labelings and $(a,d)-$distance antimagic labelings for some families of graphs containing one or more cycles: cycles, suns, complete graphs, prisms, wheels, fans, and friendship graphs.

\subsection{Cycle}

\noindent In \cite{AK12}, Arumugam and Kamatchi gave a characterization of $(a,d)$-distance antimagic cycles.
\begin{theorem} \emph{\cite{AK12}} The cycle $C_n$ is $(a,d)$-distance antimagic if and only if $n$ is odd and $d=1$.
\end{theorem}

\noindent The characterization missed out a single case when $n=4$ and $d=0$ and so we rewrite the theorem as follow.
\begin{theorem} A cycle $C_n$ has an $(a,d)-$distance antimagic labeling if and only if $d=0$ and $n=4$ or $d=1$ and $n$ is odd.
\label{Cnd1}
\end{theorem}

\noindent The previous theorem showed that only odd cycles have $(a,d)$-distance antimagic labelings for $d\geq 1$. However in the next theorem we shall construct distance antimagic labelings for even cycles.

\begin{theorem}
All cycles are distance antimagic.
\label{Cn}
\end{theorem}
\begin{proof} Consider a cycle $C_n$ of order $n$. For odd $n$, it is already $(a,d)$-distance antimagic by Theorem \ref{Cnd1}. For even $n=2k$, we define a vertex labeling $f$ as follow. Suppose that $V(C_n)=\{x_1,x_2,\cdots,x_n\}$ and $E(C_n)=\{x_nx_1, x_ix_{i+1}, i=1,2, \ldots n\}$.\\

\[f(x_i)=\left \{\begin{array}{cl}
  1, & \mbox{for}~ i =1, \\
  i-1, & \mbox{for odd}~ i, 3 \leq i \leq k+1, \\
  n+2-i, & \mbox{for odd}~ i, k+2 \leq i \leq n-1,  \\
  \frac{n}{2}-1+i, & \mbox{for even}~ i, 2 \leq i \leq k+1, \\
    \frac{3n}{2}+2-i, & \mbox{for even}~ i, k+2 \leq i \leq n.
  \end{array}
\right.
\]
Under the previous labeling, we obtain the following all distinct vertex-weights.
\[w(x_i)=\left \{\begin{array}{cl}
  n+3, & \mbox{for}~ i =1, \\
  n-2+2i, & \mbox{for odd}~ i, 3 \leq i \leq k, \\
  2n-1, & \mbox{for odd}~ i = k+1 \mbox{or}~ k+2,\\
  3n+4-2i, & \mbox{for odd}~ i, k+3 \leq i \leq n-1,\\
  3, & \mbox{for}~ i = 2, \\
  2i-2, & \mbox{for even}~ i, 4 \leq i \leq k, \\
  n-1 + \frac{1+i}{2}, & \mbox{for even}~ i = k+1 \mbox{or}~ k+2,\\
  2n+4-2i, & \mbox{for even}~ i, k+3 \leq i \leq n.\\
  \end{array}
\right.
\]
\end{proof}

\noindent Adding an edge to each vertex in a cycle results in a unicyclic graph called sun. While for cycle, $(a,d)$-distance antimagic labelings do not exist for even cycles; for suns, the labelings do not exist for all suns.

\subsection{Sun}

\emph{A sun $S_n$} is a cycle on $n$ vertices with a leaf attached to each vertex on the cycle. Let the vertex set of sun $V(S_n)=\{x_1, \ldots, x_n, y_1, \ldots,y_n\},$ where $d(x_i)=3$ and $d(y_i)=1$.

\begin{theorem}
All suns are not $(a,d)$-distance antimagic.
\end{theorem}
\begin{proof} Consider a sun $S_n$ of order $2n$. Since $w(y_i)=f(x_i)$ then $1 \leq w(y_i) \leq 2n$ and so $d \leq \frac{2n-1+1}{n} = 2$. For $d=0$, it is obvious that a distance magic labeling does not exist. If $d=1$ then the labels of $x_i$s are $c,c+1,\ldots,c+n-1$ for $1 \leq c \leq n+1$. Thus the smallest possible weight of $x_i$ is $c+(c+1)+(c+n)=3c+n+1$ and so there is a gap in vertex-weights. If $d=2$ then the labels of $x_i$ are either $1,3,\ldots,2n-1$ or $2,4,\ldots,2n$. In both cases, there will be parity difference between $w(x_i)$ and $w(y_i)$.
\end{proof}

\noindent Although all suns are not $(a,d)$-distance antimagic, next we shall prove that they are otherwise distance antimagic.

\begin{theorem}
All suns are distance antimagic.
\end{theorem}
\begin{proof} We define a vertex labeling $f$ of $S_n$ as follow.
\[f(x_i)=n+i \quad \mbox{for}\quad i=1,2,\ldots,n,\quad \mbox{and}\]
\[f(y_i)=i \quad \mbox{for}\quad i=1,2,\ldots,n,\]
and so
\[w(y_i)=f(x_i) = n+i \quad \mbox{for}\quad i=1,2,\ldots,n,\quad \mbox{and}\]
\[\begin{array}{ll}
w(x_{i})&= \left\{
    \begin{array}[c]{llll}%
    3n+3 & \mbox{for}\quad i=1,\\
    2n+3i & \mbox{for}\quad i= 2,3, \ldots, n-1,\\
    4n & \mbox{for}\quad i= n.
    \end{array}\right.\end{array}\]
\noindent When $n \neq 0 \mod 3$, all vertex-weights are distinct. Otherwise, $3n+3=2n+3i$ for $i=\frac{n}{3}+1$ and $4n=2n+3i$ for $i=\frac{2n}{3}$. In that case, we exchange the labels of $y_{\frac{n}{3}+1}$ with $y_{\frac{n}{3}}$ and $y_{\frac{2n}{3}}$ with $y_{\frac{2n}{3}+1}$ to obtain distinct weights for all vertices.
\end{proof}

\noindent We have studied the distance antimagic labelings for cycles, the 2-regular connected graphs, and next we will consider two families of regular connected graphs: prisms and complete graphs. Here we manage to characterize all $(a,d)$-distance antimagic prisms and complete graphs.

\subsection{Prism}

\noindent \emph{A prism $C_n \times P_2$} is a 3-regular graphs of order $2n$. Let $V(C_n \times P_2)=\{x_1, \ldots, x_n, y_1, \ldots, y_n\}$ and  $E(C_n \times P_2)=\{x_i y_i, i=1, \ldots, n\}.$\\

\noindent In \cite{AK12}, Arumugam and Kamatchi proved that prisms are $(a,1)$-distance antimagic.
\begin{theorem} \emph{\cite{AK12}} The prism $C_n \times K_2$ is $(n+2,1)$-distance antimagic.
\label{prism1}
\end{theorem}

\noindent Next we will prove that $(a,d)$-distance antimagic prisms only exist when $d=1$.
\begin{theorem}
A prism is $(a,d)$-distance antimagic if and only if $d=1$.
\end{theorem}
\begin{proof} By Lemma \ref{ubd}, $d \leq 3 \frac{2n-3}{2n-1} \leq 2$. If $d=0$ or $d=2$, then $a=\frac{3(2n+1)-d(2n-1)}{2}$ is not an integer. By Theorem \ref{prism1}, we have the desired labeling for $d=1$.\end{proof}



\subsection{Complete Graph}

\begin{theorem}
A nontrivial complete graph has an $(a,d)$-distance antimagic labeling if and only if $d=1$.
\label{Kn}
\end{theorem}
\begin{proof} Consider a complete graph of order $n$, $K_n$. Since $K_n$ is $(n-1)$-regular, then by Lemma \ref{ubd}, $d \leq 1$. It is known that distance magic labelings do not exist for nontrivial complete graphs (see for example \cite{MRS03}), and so $d=1$. Again, by applying Lemma \ref{ubd}, we obtain $a=\frac{(n-1)n}{2}$.\\

\noindent Suppose that $V(K_n)=\{x_1, x_2, \ldots ,x_n\}$. We then define a vertex labeling of $K_n$ as follow.
\[f(x_i)=i \quad \mbox{for}\quad i= 1,2,\ldots,n.\]
Under the labeling $f$, the vertex-weights are
\[w(x_i)=\frac{n(n+1)}{2}-i \quad \mbox{for}\quad i= 1,2,\cdots,n,\]
which constitute an arithmetic progression with difference 1.
\end{proof}

\noindent The last families of graphs to be considered are wheels, fans, and friendship graphs. They are closely related since deleting one edge in a wheel results in a fan and deleting half of the edges results in a friendship graph. Not surprisingly, the graphs have similar distance antimagicness characteristics. We prove that all three graphs are not $(a,d)$-distance antimagic in general, but distance antimagic instead.

\begin{figure}[h]
\centerline{\includegraphics[width=0.9\textwidth]{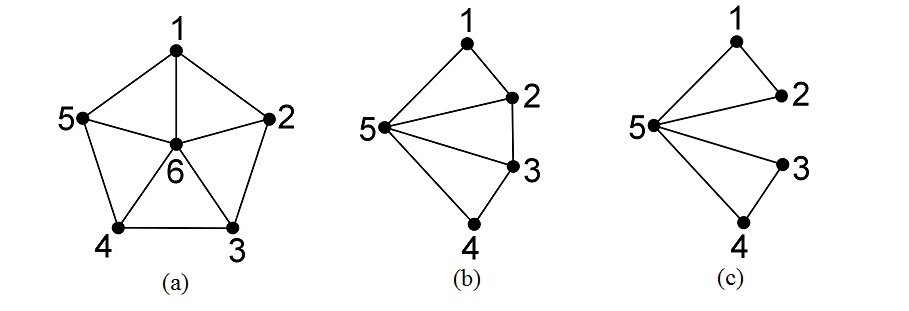}}
\caption{\small $(a,d)$-distance antimagic labelings for wheel-related graphs.}
\protect\label{figure}
\end{figure}

\subsection{Wheel}

\noindent \emph{A wheel $W_n$} is a graph obtained by joining all vertices of a cycle of order $n$ to a further vertex called the
\emph{center}. Let $V(W_n)=\{x_0,x_1, \ldots, x_n\}$ where $v_0$ is the center and $x_1, \ldots, x_n$ are the vertices of the cycle.

\begin{lemma} A wheel $W_n$ of order $n+1$ has an $(a,d)-$distance antimagic labeling if and only if $3 \leq n \leq 5$.
\end{lemma}
\begin{proof}
\noindent Since $d(x_i)=3$ then $6 \leq w(x_i) \leq 3n$, for $i=2,\cdots, n-1$. Thus, $a+(n-1)d \leq 3n$ or $d \leq 3 - \frac{3}{n-1}$, and so $d \leq 2$. On the other hand,
$\frac{n(n+1)}{2} \leq w(x_0) \leq \frac{n(n+3)}{2}$. This leads to $w(x_0) - w(x_i) \geq \frac{n(n+1)}{2} - 3n$ for some $i$ and so $d \geq \frac{n^2-5n}{2}$. For $n \geq 6$, we obtain $d \geq 3$, a contradiction.\\

\noindent Now we need to consider $W_n$ for $n=3,4,5$. For $n=3$, since $W_3 \simeq K_4$, we have the desired labeling as in Theorem \ref{Kn}. For $n=4$, it is known that $W_4$ is distance magic or $(10,0)$-distance antimagic (see \cite{MRS03}). For $d>0$, by Lemma \ref{not-da}, $W_4$ is not $(a,d)$-distance antimagic. To complete the proof, for $n=5$, consider a vertex labeling of $W_5$ whose vertex-weights constitute an arithmetic progression with difference 1 as depicted in Figure \ref{figure}(a).
\end{proof}

\noindent Although only two small wheels are $(a,d)$-distance antimagic for $d\geq 1$, we could construct distance antimagic labelings for all wheels of order other than 5.
\begin{lemma}
All wheels of order other than 5 are distance antimagic.
\label{Wn}
\end{lemma}
\begin{proof}
By Lemma \ref{not-da}, $W_4$ is not distance antimagic. For $n \neq 4$, define a vertex labeling where for $i=1, \ldots, n$, $x_i$ is labeled as vertex $x_i$ of a cycle $C_n$ in the proof of Theorem \ref{Cn} and $v_0$ is labeled with $n+1$. Since the vertex-weights of vertices in the cycle are distinct, then the vertex-weights of vertices in the wheel are also distinct.
\end{proof}


\subsection{Fan}
\emph{A fan $F_n$} is a graph obtained by joining all vertices of a path of order $n$ to a further vertex called the
\emph{center}. Let $V(F_n)=\{x_0,x_1, \ldots, x_n\}$ where $x_0$ is the center and $x_1, \ldots, x_n$ are the vertices of the path.

\begin{theorem}
The fan $F_{n}$ admits an $(a,d)-$distance antimagic labeling if and only if $n = 2$ or $n = 4$.
\end{theorem}
\begin{proof} Since $d(x_1)=d(x_n)=2$ and $d(x_i)=3$ for $i=2,\cdots, n-1$ then $3 \leq w(x_i) \leq 3n$. Thus, $d \leq \frac{3n-3}{n}$, and so $d \leq 2$. On the other hand, $\frac{n(n+1)}{2} \leq w(x_0) \leq \frac{n(n+3)}{2}$. This leads to $w(x_0) - w(x_i) \geq \frac{n(n+1)}{2} - 3n$ for some $i$ and so $d \geq \frac{n^2-5n}{2}$. For $n \geq 6$, $d \geq 3$, a contradiction.

\noindent Now we need to consider $F_n$ with $2 \leq n \leq 5$. For $n=2$, since $F_2 \simeq K_3$, we have the desired labeling as in Theorem \ref{Kn}. For $n=3$, $F_3$ has no $(a,d)$-distance antimagic labeling by Lemma \ref{not-Dda}. For $n=4$ consider a vertex labeling of $F_4$ whose vertex-weights constitute an arithmetic progression with difference 1 as depicted in Figure \ref{figure}(b). For $n=5$, if we assign $1, 2, 3$ or $6$ as the label of vertex $x_0$ then the difference between the weight of $x_0$ and the largest weight of $x_i, i=1, \ldots, n$ is greater than 2, a contradiction. If we assign $4$ as the label of $x_0$ then $w(x_0)=17$. Due to the impossibility to attain $16$ as weight, we have $d=2$ and so the weights of all $x_i$s are odd. This implies that the labels of $x_2$ and $x_4$ must be odd, causing the weight of $x_3$ to be even, a contradiction. If we assign $5$ as the label of $x_0$ then $w(x_0)=16$. If $d=1$ then $w(x_i) \geq 11$, for $i=1,2,\ldots,5$. However, the weights of $x_1$ and $x_5$ are summation of two labels, one of which is $5$, and so the weight $12$ is not achievable. If $d=2$ then the weights of all $x_i$s are even; thus the labels of $x_2$ and $x_4$ must be odd and the weight of $x_3$ is also odd, a contradiction.
\end{proof}

%

\noindent Again, we could prove that all fans, except for $F_3$, are distance antimagic.
\begin{lemma}
All fans of order other than 4 are distance antimagic.
\end{lemma}
\begin{proof} By Lemma \ref{not-da}, $F_3$ is not distance antimagic. For $n \neq 3$, by defining the following vertex labeling $f$
\[f(x_{i})=\left\{
    \begin{array}[c]{llll}%
    \lceil \frac{n+2}{2}\rceil & \mbox{for}\quad i=0,\\
    i & \mbox{for}\quad i= 1,2,\cdots,\lfloor \frac{n+1}{2}\rfloor, \\
    1+i & \mbox{for}\quad i= \lfloor \frac{n+1}{2}\rfloor+1, \lfloor
    \frac{n+1}{2}\rfloor+2,\cdots,n,
    \end{array}\right.\]
we obtain all distinct vertex-weights bellow.
\[w(x_{i})=\left\{
    \begin{array}[c]{llll}%
    \frac{1}{2} (n+1)(n+2)- \lceil \frac{n+2}{2}\rceil & \mbox{for}\quad i=0,\\
    \lceil \frac{n+2}{2}\rceil + 2 & \mbox{for}\quad i= 1\\
    \lceil \frac{n+2}{2}\rceil + 2i & \mbox{for}\quad i= 2,3,\cdots,\lfloor \frac{n+1}{2}\rfloor-1
                             \quad \mbox{and}\\
                             & \quad \quad ~i= \lfloor \frac{n+1}{2}\rfloor+2,\lfloor
                             \frac{n+1}{2}\rfloor+3,\cdots,n-1,\\
    2\lceil \frac{n+2}{2}\rceil + \lfloor \frac{n+1}{2}\rfloor &
                \mbox{for}\quad i=\lfloor \frac{n+1}{2}\rfloor,\\
    2\lceil \frac{n+2}{2}\rceil + \lfloor \frac{n+1}{2}\rfloor + 2
            & \mbox{for}\quad i=\lfloor \frac{n+1}{2}\rfloor+1,\\
    \lceil \frac{n+2}{2}\rceil + n & \mbox{for}\quad i=n.
    \end{array}\right.\]
\end{proof}

\subsection{Friendship graph}

A friendship graph $f_n$ is obtained by identifying a vertex from
$n$ copies of complete graphs of order 3. Let
$V(f_n)=\{x_0, x_1, \ldots, x_{2n}\}$ where $x_0, x_{2i-1}, x_{2i}$ are the vertices in the $i$-th $K_3$, for $i=1,\ldots,n$.

\begin{theorem}
A friendship graph $f_{n}$ is $(a,d)$-distance antimagic if and only if $n = 1$ or $n=2$.
\end{theorem}
\begin{proof} For $i=1,2,\ldots, 2n$, we have $3 \leq w(x_i) \leq 4n+1$ and so $d$ is at most $\frac{(4n+1)-3}{2n} = 2-\frac{1}{n} \leq 2$. On the other hand, $n(2n+1) \leq w(x_0) \leq n(2n+3)$. Thus we have $w(x_0)-w(x_i) \geq n(2n+1) - (4n+1) = 2n^2-3n-1$. For $n \geq 3$, $w(x_0)-w(x_i) \geq 8$, a contradiction.\\

\noindent To complete the proof, we need to consider $f_1$ and $f_2$. Since $f_1\simeq K_3$ then $f_1$ has a
$(3,1)$-distance antimagic labeling by Theorem \ref{Kn}. A $(6,1)$-distance antimagic labeling for $f_2$ is depicted in Figure \ref{figure}(c).
\end{proof}

%

\noindent Finally, a simple vertex labeling leads to the distance antimagicness of friendship graphs.
\begin{theorem}
All friendship graphs are distance antimagic.
\end{theorem}
\begin{proof} We define a vertex labeling $f$ of $f_n$ as
follow
\[ f(x_i)=\left\{
    \begin{array}[c]{llll}
    n+1 & \mbox{for}\quad i=0,\\
    i & \mbox{for}\quad i= 1,2,\ldots,n,
    \end{array}\right.\]
and so we obtain the following vertex-weights
\[
w(x_i) =\left\{\begin{array}[c]{lll}
    n (2n+1) & \mbox{for} & i=0,\\
    n+2+i & \mbox{for} & i=1,3,\ldots,2n-1,\\
    n+i & \mbox{for} & i=2,4,\ldots,2n.
\end{array}\right.\]
We can see that the weights are all distinct.
\end{proof}


\section{Final remark}

\noindent Revisiting the necessary conditions for the existence of distance antimagic and $D$-distance antimagic graphs in Lemmas \ref{not-da} and \ref{not-Dda}, we strongly believe that those conditions are also sufficient and propose the following conjectures.

\begin{conjecture}
A graph is distance antimagic if and only if it does not contain two vertices with the same neighborhood.
\end{conjecture}

\begin{conjecture}
A graph is $D$-distance antimagic if and only if it does not contain two vertices with the same $D$-neighborhood.
\end{conjecture}

\noindent As with the antimagic conjecture of Harstfield and Ringel \cite{HR90}, proving or disproving the afore-mentioned conjectures is likely to be a hard problem.


\end{document}